\documentclass[preprint]{elsarticle}

\usepackage{dsfont}
\usepackage[latin1]{inputenx}
\usepackage{amsfonts}
\usepackage{latexsym}
\usepackage{enumerate}
\usepackage{amsthm,amsmath,amssymb,mathrsfs,yhmath}
\usepackage{tikz}
\usepackage{fullpage}
\usepackage{wasysym}
\usetikzlibrary{positioning}
\usetikzlibrary{matrix}
\usepackage{float}
\usepackage{subfig}

\usepackage{pgfplots}
\usetikzlibrary{intersections}

\usepackage[pdftex,bookmarks=true]{hyperref}
\usepackage{url}

\newtheorem*{thm}{Theorem} 
\newtheorem{thm1}{Theorem}[section]  
\newtheorem{lemma}[thm1]{Lemma}   

\newtheorem*{corollary}{Corollary}

\theoremstyle{remark}
\newtheorem*{acknowledgements}{Acknowledgements}

\theoremstyle{definition}
\newtheorem{dfn}[thm1]{Definition}
\newtheorem{exam}[thm1]{Example}
\newtheorem{rem}[thm1]{Remark}

\usepackage[pdftex,bookmarks=true]{hyperref}
\usepackage{url}

\newcommand{\cP}{\mathcal{P}}

\newcommand{\TT}{\mathbb{T}}

\newcommand{\abs}[1]{\lvert #1\rvert}

\newcommand{\Hh}[1]{\operatorname{dim}_k\operatorname{HH}^{#1}(A_\TT)}
\newcommand{\HH}[1]{\operatorname{dim}_k\operatorname{HH}^{#1}(A)}
\newcommand{\Par}[1]{^-(0,0)^-_{#1}}
\newcommand{\Ge}[1]{\operatorname{dim}_k k\mathcal G_{#1}/ \operatorname{Im}(1-t)}

\begin{document}

\begin{frontmatter}
\title{Hochschild cohomology of Jacobian algebras from unpunctured surfaces: A geometric computation}

\author{Yadira Valdivieso-D\'iaz}
\ead{valdivieso@mdp.edu.ar}


\address{Departamento de Matem\'atica, Facultad de Ciencias Exactas y Naturales, Funes 3350, Universidad Nacional de Mar del Plata, 7600 Mar del Plata, Argentina.}

\begin{abstract}
There are several examples in which algebraic properties of Jacobian algebras from (unpunctured) Riemann surfaces can be computed from the geometry of the Riemann surface. 

In this work, we compute the dimension of the Hochschild cohomology groups of any Jacobian algebra from unpunctured Riemann surfaces. In those expressions appear geometric objects of the triangulated surface, namely: the number of internal triangles and certain types of boundaries. Moreover, we give geometric conditions on the triangulated surface $(S,M,\TT)$ such that the Gerstenhaber algebra $\operatorname{HH}^*(A_\TT)$ has non-trivial multiplicative structures.

We also show that the derived class of Jacobian algebras from an unpunctured surface $(S,M)$ is not always completely determined by the Hochschild cohomology.
\end{abstract}


\end{frontmatter}


\section{Introduction}

Let $k$ be an algebraically closed field. A potential $W$ for a quiver $Q$ is a possibly infinite linear combination of cyclic paths in the  complete path algebra $k\langle\langle Q\rangle\rangle$. The Jacobian algebra $\cP(Q,W)$ associated to a quiver with a potential
$(Q,W)$ is  the quotient of the complete path algebra $k\langle\langle Q\rangle\rangle$ modulo the Jacobian
ideal $J(W)$. Here, $J(W)$ is the topological closure of the ideal of $k\langle\langle Q\rangle\rangle$
which is generated by the cyclic derivatives of $W$ with respect to the arrows
of $Q$.

For any (tagged) triangulation $\TT$ of a Riemann surface with marked points $(S,M)$, it is possible to construct a finite dimensional Jacobian algebra $A_\TT$, whose quiver $Q_\TT$ was defined by Fomin, Shapiro and Thurston in \cite{FST08} and the potential $W_\TT$ was given by Labardini-Fragoso in \cite{LF09,LF12}. In case $(S,M)$ is a unpunctured surface, the Jacobian algebra $A_\TT$ is isomorphic to the gentle algebra defined by Assem, Br\"ustle, Charbonneau-Jodoin and Plamondon in \cite{ABCJP10}; see Section 2 for definition.

The Hochschild cohomology groups $\operatorname{HH}^{i}(A)$ of an algebra $A$, where $i\geq 0$, were introduced by Hochschild in \cite{Ho46}. The low-dimensional groups, namely for $i=0,1,2$, have a concrete interpretation of classical algebraic structures, but in general it is quite hard to compute them, even in those lower dimensional cases. However, an explicit formula for the dimension of the Hochschild cohomology $\HH{i}$ of some subclasses of special biserial algebras had been computed in terms of combinatorial data, for example in \cite{B06,ST10,Lad12b}. In particular, recently Redondo and Rom\'an obtained a formula for quadratic string algebras and therefore for gentle algebras (see \cite{RR15}).

The aim of this work is to compute the dimension of $\operatorname{HH}^n(A_\TT)$ for any triangulated surface $(S, M, \TT)$ in terms of geometric data using the results of Redondo and Rom\'an. As a consequence of these computations, we have that the non-trivial multiplicative structure of the Gerstenhaber algebra $\operatorname{HH}^*(A_\TT)$ depends on the existence of internal triangles in $\TT$, see Section 2 for definition of internal triangle. Our main result is the following:

\begin{thm}\label{principal}
Let $A_\TT=\cP(Q_\TT, W_\TT)$ be the Jacobian algebra of an unpunctured triangulated surface $(S,M,\TT)$. Denote by $\operatorname{int}(\TT)$ the set of internal triangles of $\TT$, by $\mathcal B_0$ the set of boundaries of type 0 and by $\mathcal B_1$ the set of boundaries of type 1(See Figure \ref{Frontera0} and \ref{Frontera} in Section 3). Then:

\begin{itemize}
\item[(a)] $\Hh{0}=1+\abs{\mathcal B_0}$
\item[(b)] $\Hh{1}=1+ \abs{\mathcal B_1}+\abs{(Q_\TT)_1}-\abs{(Q_\TT)_0}$
\item[(c)] If $\operatorname{char}k\neq 2$ and $n\geq 2$, then
$$\Hh{n}=
\begin{cases}
\abs{\operatorname{int}(\TT)} & \mbox{ if \ } n\equiv 0,1 \pmod{6}\\
0 & \mbox{otherwise}
\end{cases}$$
\item[(d)] If $\operatorname{char}k= 2$ and $n\geq 2$, then
$$\Hh{n}=
\begin{cases}
\abs{\operatorname{int}(\TT)} & \mbox{ if \ } n\equiv 0,1 \pmod{3}\\
0 & \mbox{otherwise}
\end{cases}.$$
\end{itemize}
Moreover, if the triangulation $\TT$ contains at least one internal triangle, then the cup product defined in $\oplus_{k\geq0}\operatorname{HH}^k(\Lambda)$ is non trivial; and if $\operatorname{char}k=0$, then the Lie bracket is also non trivial.
\end{thm}

\begin{rem}\label{HH}
We can compute the number of arrows of  $Q_\TT$ from the triangulated surface $(S,M,\TT)$. It is enough to consider the number of internal triangles and the number of triangles with exactly one arc being a boundary segment. More precisely, if we denote by $\operatorname{SInt}(\TT)$ the set of of triangles with exactly one arc being part of a boundary component, then $$\abs{(Q_\TT)_1}=3\abs{\operatorname{Int}(\TT)}+\abs{\operatorname{SInt}(\TT)}.$$
Since the number of vertices of $Q_\TT$ can be also computed from $(S,M,\TT)$, then $\Hh{1}$ can be computed from the geometry of $(S,M,\TT)$, with the following expression:$$\Hh{1}=1+ \abs{\mathcal B_1}+3\abs{\operatorname{Int(\TT)}}+\abs{\operatorname{SInt}(\TT)}-6g-3b-c+6,$$ where $g$ is the genus of $S$, $b$ is the number of boundary components and $c$ the number of marked points.
\end{rem}

Rickard in  \cite{Ric91}[Proposition 2.5] showed that the Hochschild cohomology is invariant under derived equivalence (see also \cite{Ha89}[Theorem 4.2] for a special case). But it is not always true that the Hochschild cohomology determine the derived class of a family of algebras. In the case of Jacobian algebras from an unpunctured disc, it was proven by Bustamente and Gubitosi in \cite{BG14}. that the Hochschild cohomology and the number of vertices determine the derived classes.

The main result show that the derived class of Jacobian algebras from a unpunctured surface $(S,M)$ is not always completely determined by the Hochschild cohomology. We give an example of such affirmation.

Also, as consequence of the result we obtain another proof of the computation of Ladkani restricted to the algebras we are working with, which is given in \cite{Lad12b}.

\section{Preliminaries}

In this section, we recall definitions and concepts used in this work. We also fix notations.

\subsection{Quivers and relations}

Let $Q$ be a finite quiver with a set of vertices $Q_0$, a set of arrows $Q_1$ and let  $s,t:Q_1 \to Q_0$ be the maps associating to each arrow $\alpha$ its source $s(\alpha)$ and its target $t(\alpha)$. A path $w$ of length $l$ is a sequence of $l$ arrows $\alpha_1, \dots, \alpha_l$ such that $t(\alpha_i)=s(\alpha_{i+1})$, we say that its source $s(w)$ is $s(\alpha_1)$ and its target $t(w)$ is $t(\alpha_l)$. We denote by $\abs{w}$ the length of the path $w$. 

The path algebra $kQ$ is the $k$-vector space with basis the set of paths in $Q$ and the product of the basis elements is given by the concatenations of the sequences of arrows of the paths $w$ and $w'$ if they form a path and zero otherwise. Let $F$ be the two-sided ideal of $kQ$ generated by the arrows of $Q$. A two-sided ideal $I$ is said to be admissible if there exists an integer $m\geq 2$ such that $F^m\subseteq I \subseteq F^2$ and its elements are called \emph{relations}. The pair $(Q,I)$ is called a \emph{bounded quiver}.

The quotient algebra $kQ/I$ is said a \emph{monomial algebra} if the admissible ideal $I$ is generated by paths and a relation is called \emph{quadratic} if it is a path of length two.

\subsection{Surfaces and triangulations}
Let $S$ be a connected oriented Riemann surface with  non-empty boundary $\partial S$ and let $M$ be a non-empty finite set on the boundary $\partial S$ such that there is at least one marked point in each boundary component. The elements of $M$ are called \emph{marked points}. We will refer to the pair $(S,M)$ as an \emph{unpunctured surface}.

\begin{dfn}
An \emph{arc} $\tau$ in $(S,M)$ is a curve in $S$ such that
\begin{itemize}
\item[(a)] the endpoints are in $M$;
\item[(b)] $\tau$ does not cross itself;
\item[(c)] $\tau$ is not isotopic to a point or to a boundary segment;
\item[(d)] $\tau$ does nor cut a monogon (disc with one marked point) or a digon (disc with two marked points).
\end{itemize}
\end{dfn}

For any two arcs $\tau$ and $\tau'$ in $S$, let $e(\tau, \tau')$ be the minimal number of crossings of $\tau$ and $\tau'$, that is, $e(\tau, \tau')$ is the minimum of numbers of crossings of curves $\sigma$ and $\sigma'$, where $\sigma$ is isotopic to $\tau$ and $\sigma$ is isotopic to $\tau'$. Two arcs $\tau$ and $\tau'$ are called \emph{non-crossing} if $e(\tau,\tau')=0$. A \emph{triangulation} $\TT$ is a maximal collection of non-crossing arcs. The arcs of a triangulation $\TT$ cut the surface into \emph{triangles}. A triangle $\triangle$ in $\TT$ is called an \emph{internal triangle} if none of its sides is a boundary segment. We refer to the triple $(S,M,\TT)$ as a \emph{triangulated surface}.

\subsection{Jacobian algebras from surfaces}
If $\TT=\{\tau_1, \cdots\tau_n\}$ is a triangulation of an unpunctured surface $(S,M)$, we define a quiver $Q_\TT$ as follows: $Q_\TT$ has $n$ vertices, one for each arc in $\TT$. We will denote the vertex corresponding to $\tau_i$ by $e_i$ (or $i$ if there is no ambiguity). The number of arrows from $i$ to $j$ is the number of triangles $\triangle$ in $\TT$ such that the arcs $\tau_i, \tau_j$ form two sides of $\triangle$, with $\tau_j$ following $\tau_i$ when going around the triangle $\triangle$ in the counter-clockwise orientation. Note that the interior triangles in $\TT$ correspond to oriented 3-cycles in $Q_{\TT}$.

Following \cite{ABCJP10,LF09}, let $W_\TT$ be the sum of all oriented 3-cycles in $Q_\TT$. Then $W_\TT$ is a potential which gives rise to a Jacobian algebra $A_\TT=\cP(Q_\TT, W_\TT)$. The latter, in this case, is isomorphic to the quotient of the path algebra of the quiver $Q_\TT$ by the two-sided ideal generated by the subpaths of length two of each oriented 3-cycle in $Q_\TT$ (see \cite{DWZ08} for definitions of quivers with potential).
 
\subsection{Dimension of the Hochschild cohomology groups of gentle algebras}

Recall that a finite dimensional algebra $A$ is \emph{gentle} if it admits a presentation $kQ/I$ satisfying the following conditions:

\begin{itemize}
\item[(G1)] At each point of $Q$ start at most two arrows and stop at most two arrows.
\item[(G2)] The ideal $I$ is generated by paths of length  2.
\item[(G3)] For each arrow $\beta$ there is at most one arrow $\alpha$ with $s(\beta)=t(\alpha)$ and at most one arrow $\gamma$  with $t(\beta)=s(\gamma)$ such that $\alpha\beta\in I$ and $\beta\gamma\in I$.
\item[(G4)] For each arrow $\beta$ there is at most one arrow $\alpha$ with $s(\beta)=t(\alpha)$ and at most one arrow $\gamma$  with $t(\beta)=s(\gamma)$ such that $\alpha\beta\notin I$ and $\beta\gamma\notin I$.
\end{itemize}

The Hochschild cohomology groups of gentle algebras have already been computed in \cite{Lad12b} by Ladkani and in \cite{RR15} by Redondo and Rom\'an. In the first case,  these results have been expressed in terms of the derived invariant introduced by Avella-Alaminos and Geiss in \cite{AG08}.  In the second one, Redondo and Rom\'an used Bardzell's resolution (see \cite{Ba97}). The results of Redondo and Rom\'an are more general, since they also computed the Hochschild cohomology of quadratic string algebras and found conditions on the bound quiver associated to the string algebra in order to have non-trivial multiplicative structures (see \cite{SW83} for definition of string algebras).

Since Jacobian algebras from unpunctured surfaces are gentle, we use the computation of Redondo and Rom\'an to prove our main result. We include the statement of their result in this work for convenience to the reader, but first we need to introduce some notations related to Bardzell's resolution, as well as others used in \cite{RR15}.

Let $A=kQ/I$ be a monomial algebra, we fix a minimal set $\mathcal R$ of paths that generates the ideal $I$. Moreover, we fix a set $\mathcal P$ of paths in $Q$ such that the set $\{\gamma+I \mid \gamma\in \mathcal P\}$ is a basis of $A$.

Since gentle algebras are monomial algebras, their Hochschild cohomology groups can be computed using a convenient projective resolution of $A$ as $A$-$A$-bimodule constructed by Bardzell in \cite{Ba97}. In this particular case, the \emph{Hochschild complex}, obtained by applying $\operatorname{Hom}_{A-A}(-, A)$ to the Hochschild resolution of Bardzell and using the isomorphisms 

$$\operatorname{Hom}_{A-A}(A\otimes_{E} kAP_n\otimes_{E} A, A)\cong \operatorname{Hom}_{E-E}(kAP_n,A)$$
is 

$$0 \longrightarrow \operatorname{Hom}_{E-E}(kAP_0, A) \stackrel{F_1}{\longrightarrow} \operatorname{Hom}_{E-E}(kAP_1, A) \stackrel{F_{2}}{\longrightarrow}\operatorname{Hom}_{E-E}(kAP_2, A)\stackrel{F_{3}}{\longrightarrow}\cdots$$
where 
\begin{itemize}
\item[(i)] $E=kQ_0$ is the sub-algebra generated by the vertices of $Q$,
\item[(ii)]$kAP_0$ is the vector space $kQ_0$, $kAP_1$ is the vector space $kQ_1$ and $kAP_n$ is the one generated by the set $$AP_n=\{\alpha_1\alpha_2\cdots\alpha_n\mid \alpha_i\alpha_{i+1}\in I, \ 1\leq i<b\}$$ for $n\geq 2$, and
\item[(iii)] the morphisms are $$F_1(f)(\alpha)=\alpha f(e_{t(\alpha)})- f(e_{s(\alpha)})$$
$$F_n(f)(\alpha_1\cdots\alpha_n)=\alpha_1f(\alpha_2\cdots \alpha_n)+(-1)^nf(\alpha_1\cdots\alpha_{n-1})\alpha_n.$$
\end{itemize}

It is easy to show the $k$-vector spaces $\operatorname{Hom}_{E-E}(kAP_n, A)$ and the one generated by the set of pairs of parallel paths
$$(AP_n// \mathcal P)=\{(\rho, \gamma)\in AP_n \times \mathcal P\mid s(\rho)=s(\gamma), t(\rho)=t(\gamma)\}.$$
are isomorphic. Then the Hochschild cohomology of monomials algebras can be computed using $k(AP_n// \mathcal P)$. For gentle algebras, it is enough to consider the following subsets of $(AP_n// \mathcal P)$:

\begin{enumerate}
\item[(a)] $^-(Q_0 // \mathcal P_1)^-=\{(e_r, \gamma)\in (Q_0 // \mathcal P_1)\ \mid Q_1\gamma\subset I, \ \gamma Q_1\subset I\}.$

\item[(b)] $^-(0,0)^-_1=\{(\alpha, \gamma)\in (Q_1// \mathcal P)\ \mid \gamma \notin \alpha kQ\cup kQ\alpha, Q_1\gamma\subset I, \ \gamma Q_1\subset I\}.$
\end{enumerate}
For any $n\geq 2$,
\begin{enumerate}
\item[(c)] $^-(0,0)^-_n=\{(\alpha_1\cdots\alpha_n, \gamma)\in (AP_n// \mathcal P)\ \mid \gamma \notin \alpha_1 kQ\cup kQ\alpha_n, Q_1\gamma\subset I, \ \gamma Q_1\subset I\}$.

\item[(d)] The set of \emph{complete pairs} $\mathcal C_n=\{(\alpha_1\cdots\alpha_n, e_r)\in (AP_n// \mathcal Q_0)\mid \alpha_n\alpha_1\in I \}$.

Let  $\mathbb Z_n=\langle t\rangle$ be the cyclic group of order $n$. Observe that $\mathbb Z_n$ acts on $\mathcal C_n$, with action given by $$t(\alpha_1\cdots\alpha_n, e_{s(\alpha_1)})=(\alpha_n\alpha_1\cdots\alpha_{n-1}, e_{s(\alpha_n)}).$$
For any $(\rho, e)\in\mathcal C_n$, we define its \emph{order} as the least natural number $k$ such that $t^k(\rho, e)=(\rho, e)$.

\item[(e)] The set of \emph{incomplete pairs} $\mathcal I_n=\{(\alpha_1\cdots\alpha_n, e_r)\in (AP_n// \mathcal Q_0)\mid \alpha_n\alpha_1\notin I \}$.

Observe that by definition $(AP_n// \mathcal Q_0)=\mathcal C_n\sqcup\mathcal I_n$, where $\sqcup$ denotes disjoint union.

\item[(f)] $\mathcal C_n(0)=\{(\alpha_1 \cdots \alpha_n, e_r)\in \mathcal C_n\mid \text{it does not exist }\gamma\in Q_1\setminus\{\alpha_n\} \text{ nor } \beta\in Q_1\setminus\{\alpha_1\}\text{ with } \alpha_n\beta, \gamma\alpha_1\in I\}$
%
%
\item[(g)] The set of \emph{gentle complete pairs} $$\mathcal G_n=\{(\alpha_1\cdots\alpha_n,e_r)\in \mathcal C_n\ \mid  t^m(\alpha_1\cdots\alpha_n,e_r)\in \mathcal C_n(0) \text{ for any } m\in \mathbb Z\}.$$

Denote by $k\mathcal G_n$ the $k$-vector space generated by the elements of $\mathcal G_n$.

\item[(h)] The set of \emph{empty incomplete pairs} $$\mathcal E_n=\{(\alpha_1\cdots\alpha_n,e_r) \in \mathcal I_n\mid \text{there is no relation } \beta\gamma\in I \text{ with } t(\beta)=r=s(\gamma)\}$$
\end{enumerate}


Using the previous subsets of pairs of parallel paths, the following Theorem give us the dimension of each Hochschild cohomology group and also conditions such that the Lie bracket and the cup product defined in $\oplus_{k\geq0}\operatorname{HH}^k(A)$ being non-trivial.

\begin{thm1}\cite{RR15}[Corollary 3.11, Theorem 4.5, Theorem 4.9]\label{TeoLM}
Let $A=kQ/ I$ be a gentle algebra. Then

\begin{eqnarray*}
\HH{0} & = & 1 + \abs{^-(Q_0 // \mathcal P_1)^-} \\
\HH{1} & = & \begin{cases}1+ \abs{\Par{1}}+\abs{Q_1}-\abs{Q_0} & \mbox{if } \operatorname{char}k \neq 2\\ 1+ \abs{\Par{1}}+\abs{Q_1}-\abs{Q_0} +\abs{(Q_1//Q_0)} &\mbox{if }\operatorname{char}k = 2\end{cases}\\
\HH{n} & = &\abs{\Par{n}}+\abs{\mathcal E_n} +a \Ge{n} +b\Ge{n-1}
\end{eqnarray*}

where

$$(a,b)=\begin{cases} (1,0) & \mbox{if } n\geq 2, \ n \mbox{ even}, \operatorname{char} k \neq 2\\
(0,1) & \mbox{if } n\geq 2, \ n \mbox{ odd}, \operatorname{char} k \neq 2\\
(1,1)& \mbox{if } n\geq 2,  \operatorname{char} k=2
\end{cases}
$$

Moreover, if $\mathcal G_n\neq \emptyset$ for some $n>0$, then the cup product defined in $\oplus_{k\geq0}\operatorname{HH}^k(\Lambda)$ is non trivial; and if $\operatorname{char}k=0$, then the Lie bracket is also non trivial.
\end{thm1}

\section{A geometric computation of the dimension of $\Hh{n}$}
In this section we analyze the subsets of $(AP_n //\mathcal P)$ appearing in the computation of Redondo and Rom\'an. We show that some of them have a nice geometrical description while other are zero.

\begin{lemma}\label{boundaries0}
Let $(S,M,\TT)$ be a triangulated surface. Then the set
 $$^-(Q_0 // \mathcal P_1)^-=\{(e_r, \gamma)\mid \text{$\gamma$ is a path of length at least 1 and } Q_1\gamma\subset I, \gamma Q_1\subset I\}$$
 is in bijection with the boundaries of the form in Figure \ref{Frontera0}

\begin{figure}[ht!]
\centering
\begin{tikzpicture}
\tikzset{every node/.style={fill=white}};
\draw (0,-0.2) ellipse (0.5cm and 0.8cm);
\draw (0,-1) to  [out=330,in=0, looseness=1.5] node[midway]{$\tau_r$}  (0,1.3);
\draw (0,-1) to [out=210, in=190, looseness=1.5] (0,1.3);
\filldraw [black] (0,-1) circle (1.5pt)
					(0,0.6) circle (1.5pt);
\draw (0,-1) to (1.5,-2)
		(0,-1) to (-0.2,-2)
		(0,-1) to (-1.5,-2);
\node at (.4, -1.8){$\cdots$};
\end{tikzpicture}
\caption{Boundary of type 0}
\label{Frontera0}
\end{figure}
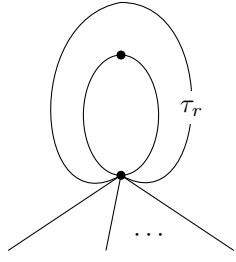
\end{lemma}

\begin{proof}
By definition, given a pair $(e_r, \gamma_1\cdots\gamma_n)\in \ ^-(Q_0//\mathcal P_1)^-$ we have that $\gamma_1\cdots\gamma_n$ is a non-zero oriented cycle starting at $e_r$. Observe that $\gamma_n\gamma_1\in I$ because $A$ is finite dimensional and quadratic, then $\gamma_1$ and $\gamma_n$ belongs to a 3-cycle arising from a triangle $\triangle$ as in Figure \ref{triangle}.

\begin{figure}[ht!]
\centering
\begin{tikzpicture}
{\tikzset{every node/.style={fill=white}};
\draw (-2,0) -- (2,0) -- (0,2)-- node[midway]{$\tau_r$} (-2,0);}
\draw(-0.1,0.1) edge[->] node[left]{\tiny$\gamma_n$}(-0.8,0.8);
\draw(-0.8,0.95)edge[->] node[above]{\tiny$\gamma_1$}(0.6,0.95);
\filldraw [black] (-2,0) node[left]{$z$} circle (1.5pt)
					(2,0) node[right]{$y$} circle (1.5pt)
					(0,2) node[above]{$x$}circle (1.5pt);
\end{tikzpicture}
\caption{Triangle $\triangle$}
\label{triangle}
\end{figure}

Recall that any non-zero path from a vertex $e_i$ to a vertex $e_j$ in $A$ is coming from arcs attached to a marked point $w$ as in Figure \ref{path}, and each arrow is opposite to $w$. Then the marked point $z$ in the triangle $\triangle$ is actually the marked point $x$, therefore $\tau_r$ is a loop.

\begin{figure}[ht!]
\centering
\begin{tikzpicture}
\draw (0,0) ellipse (0.5cm and 0.8cm);
\filldraw [black] (0,-0.8) node[above]{$w$}circle (1.5pt);
\tikzset{every node/.style={fill=white}};
\draw (0,-0.8) to node[midway]{$\tau_j$} (1.7,-2)
		(0,-0.8) to (-0.3,-2)
		(0,-0.8) to node{$\tau_i$} (-1.7,-2);
\node at (0.45, -1.8){\tiny{$\cdots$}};
\end{tikzpicture}
\caption{Non-zero path}
\label{path}
\end{figure}
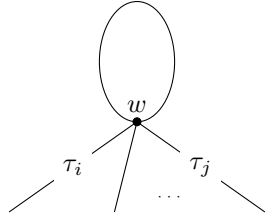

We claim that arc $\tau_r$ is part of a triangle $\triangle '$ such that the two others arcs $\tau$ and $\tau '$ are part of the boundary. Denote by $\triangle '$ the triangle formed by the arcs $\{\tau_r,\tau, \tau '\}$ different from $\triangle$. Suppose that $\tau$ is not part of the boundary, then there exists an arrow $\alpha$ such that $s(\alpha)=e_r$ or $t(\alpha)=e_r$. In the first case, $\alpha\gamma_1\cdots\gamma_m\notin I$ and in the second case $\gamma_1\cdots\gamma_m\alpha\notin I$, contradicting the fact that $(e_r, \gamma_1\cdots\gamma_n)\in \ ^-(Q_0//\mathcal P_1)^-$. Therefore no such arc exists as in Figure \ref{Frontera0}.
\end{proof}

\begin{lemma}\label{boundary1}
Let $(S,M,\TT)$ be a triangulated surface. Then the set $$^-(0,0)^-_1=\{(\alpha, \gamma)\in (Q_1// \mathcal P)\ \mid \gamma \notin \alpha kQ\cup kQ\alpha, \ Q_1\gamma\subset I, \ \gamma Q_1\subset I\}$$ is in bijection with the boundaries depicted in Figure \ref{Frontera}.

\begin{figure}[ht!]
\centering
\begin{tikzpicture}[scale=1.3]
\filldraw [black] (0,-1) circle (1.5pt)
					(0,0.33) circle (1.5pt);
\draw (0,0.3) to [out=170, in=190, looseness=1.5](0,1);
\draw (0,1) to [out=10, in=90, looseness=1.2] (0.5,0.3);
\draw (0.5,0.3) to [out=270, in=30] (0,-1);
\draw(0,0.5) circle(5pt);
\draw (0,0.33) to (-1.3,-.8);
\draw (0,0.3) to (0,-1);
\node at (-.5,-.6){$\cdots$};
\end{tikzpicture}
\caption{Boundary of type 1}
\label{Frontera}
\end{figure}
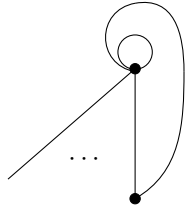

\end{lemma}

\begin{proof}
Consider a pair $(\alpha, \gamma_1\cdots\gamma_n)$ of parallel paths which belongs to $^-(0,0)_1^-$. By definition $\gamma_1$ and $\gamma_n$ are different from $\alpha$. Observe that $\alpha:i\longrightarrow j$ does not belong to a 3-cycle $\alpha\alpha_2\alpha_3$ arising from an internal triangle, otherwise $\gamma_1\cdots\gamma_n\alpha_2\in I$, but $A_\TT$ is quadratic then $\gamma_n\alpha_2\in I$. Therefore $\gamma_n$ and $\alpha_2$ belong to a 3-cycle arising from an internal triangle, then $\alpha=\gamma_n$, which is a contradiction.

Denote by $\triangle$ the triangle which gives rise to the arrow $\alpha$ and by $x, y, z\in M$ the vertices of $\triangle$ such that $x$ is opposite to $\alpha$. In Figure \ref{Gs}, we show the general structure of an arrow $\alpha$ which is parallel to a non-zero path $\gamma_1\cdots\gamma_n$.

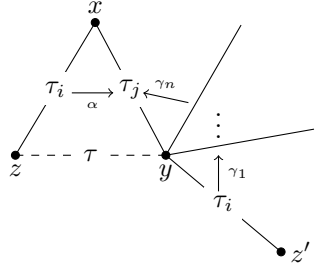
\begin{figure}[ht!]
\centering
\begin{tikzpicture}
{\tikzset{every node/.style={fill=white}};
\draw (0,0) --node[midway]{$\tau_j$} (118:2)
		(0,0) -- (60:2)
		(0,0)--(10:2)
		(0,0)--node[midway]{$\tau_i$} (320:2)
		(-2,0) --node[midway]{$\tau_i$}(118:2);
\draw [dashed] (-2,0) -- node[midway]{$\tau$}(0,0);}
\node[below] at (0,0){$y$};
\filldraw [black] (0,0) circle(1.5pt)
					(118:2) circle(1.5pt)
					(-2,0) circle(1.5pt)
					(320:2) circle(1.5pt);
\node[above] at (118:2){$x$};
\node[below] at (-2,0){$z$};
\node[right] at (320:2){$z'$};
\draw(-1.25,.83) edge[->] node[below]{\tiny$\alpha$}(-0.7,.83);
\draw (0.3,.7) edge[->] node[above]{\tiny$\gamma_n$}(-0.3, .83);
\draw(0.7, -0.48) edge[->] node[right]{\tiny$\gamma_1$}(.7,0);
\node at (0.7,.5){$\vdots$};
\end{tikzpicture}
\caption{General structure}
\label{Gs}
\end{figure}

Recall that the arrow $\gamma_1$ is opposite to the vertex $z$ because  $s(\gamma_1)=s(\alpha)$ and each arrow $\gamma_i$ is opposite to $y$ for $1\leq i\leq n$, then $y=z$, therefore the arc $\tau$ is a curve isotopic to a boundary component with a marked point $y$. Then the triangle $\triangle$ is depicted in Figure \ref{Frontera} as we claim.
\end{proof}

\begin{lemma}
The set $$^-(0,0)^-_n=\{(\alpha_1\cdots\alpha_n, \gamma)\in (AP_n// \mathcal P)\ \mid \gamma \notin \alpha_1 kQ\cup kQ\alpha_n, Q_1\gamma\subset I, \ \gamma Q_1\subseteq I\}$$ is empty for $n\geq 2$.
\end{lemma}

\begin{proof}
In order to prove the statement, we will show that given any pair of parallel paths  $(\alpha_1\cdots \alpha_n, \gamma)\in (AP_n//\mathcal P)$ such that $\gamma \notin \alpha_1 kQ\cup kQ\alpha_n$ there exists an arrow $\beta$ such that $\gamma\beta\notin I$ or $\beta\gamma\notin I$.

Denote by $\gamma_1\cdots\gamma_m$ the path $\gamma$. By hypothesis $\gamma_1\cdots\gamma_m \notin \alpha_1 kQ\cup kQ\alpha_n$, then $\gamma_1\neq \alpha_1$ and $\gamma_m\neq \alpha_n$.

Since $\alpha_i\alpha_{i+1}\in I$ for every $1\leq i\leq n-1$, the zero path $\alpha_1\cdots \alpha_n$ is a sequence of consecutive arrows of a 3-cycle $\beta_1\beta_2\beta_3$ arising form a triangle $\triangle$, such that $\alpha_1=\beta_1$. Moreover, $\alpha_{i}=\alpha_{j}=\beta_k$ if and only if $i\equiv j\equiv k\pmod{3}$, that means, in some way, that the sequence $\alpha_1\cdots\alpha_n$ is equivalent to the paths $\alpha_{1}$, $\alpha_1\alpha_2$ or $\alpha_1\alpha_2\alpha_3$, depending on the value of $n$ module $3$. We analysis three different cases: $n\equiv 0\pmod{3}$, $n\equiv 1\pmod{3}$ and $n\equiv 2\pmod{3}$.

In the first case, if $n\equiv 0\pmod{3}$, the path $\alpha_1\cdots\alpha_n$ is a cycle and $\alpha_n\alpha_1\in I$. Since $\alpha_1\neq \gamma_1$ and there is only one quadratic relation starting at $\alpha_3=\alpha_n$, namely $\alpha_3\alpha_1$, the path $\alpha_n\gamma_1\cdots\gamma_m\notin I$, as we claim.

In the second case, $n\equiv 1\pmod{3}$, we have that $n\geq 2$ and $\alpha_1=\alpha_n$. Note that there is only one quadratic relation ending at $\alpha_2$, namely $\alpha_1\alpha_2$, then $\gamma_1 \cdots\gamma_m\alpha_2\notin I$, otherwise $\gamma_m=\alpha_1$.

Finally, if  $n\equiv 2 \pmod{3}$, we have that $\beta_3\gamma_1\cdots\gamma_m\notin I$, otherwise $\beta_3\gamma_1\in I$, and that implies $\gamma_1=\beta_1=\alpha_1$ contradicting the hypothesis $\gamma \notin \alpha_1 kQ\cup kQ\alpha_n$.
\end{proof}

\begin{lemma}\label{complete}
Let $(S,M,\TT)$ be a triangulated surface. If  $n\geq 2$, then

\begin{itemize}
\item[(a)] the set $\mathcal E_n$  of empty incomplete pairs is empty.
\item[(b)] the basis of $\mathcal G_n/ \operatorname{Im}(1-t)$ is in bijection of the internal triangles of the triangulation $\mathbb T$.
\end{itemize}
\end{lemma}

\begin{proof}
In order to prove the first statement, we show first that
$$(AP_{n}//Q_0)=\begin{cases}\mathcal C_n &\mbox{ \ if \ } n\equiv 0\pmod{3}\\
\emptyset &\mbox{\ otherwise.}\end{cases}$$
Recall that any pair $(\alpha_1 \cdots \alpha_n, e_r)$ is an element of $(AP_n//Q_0)$ if and only if $\alpha_1\cdots \alpha_{n}$ is a cycle and $\alpha_i\alpha_{i+1}\in I$. It is clear that the set $(AP_2//Q_0)$ is empty because by construction $A_\TT$ has no 2-cycle.

Now, consider $n\geq 3$. By construction of $A_\TT$ any quadratic relation corresponds to the composition of two consecutive arrows of a 3-cycle arising of an internal triangle  $\triangle$ of $\TT$, then any path $\alpha_1\cdots \alpha_n$ such that $\alpha_i\alpha_{i+1}\in I$ is a sequence of consecutive arrows of a 3-cycle arising of an internal triangle $\triangle$, then $\alpha_i=\alpha_{3k+i}$ for $1\leq 3k+i\leq n$. Moreover, since the zero path $\alpha_1\cdots\alpha_n$ is a cycle, we have that $n\equiv 0\pmod{3}$ and $\alpha_n\alpha_1=\alpha_3\alpha_1\in I$, therefore $$(AP_{n}//Q_0)=\begin{cases}\mathcal C_{n} &\mbox{ \ if \ } n\equiv 0\pmod{3}\\
\emptyset &\mbox{\ otherwise.}\end{cases}$$


Since $\mathcal E_n\subseteq \mathcal I_n\subseteq (AP_n// Q_0)$ and $(AP_n// Q_0)=\mathcal C_n \sqcup \mathcal I_n$ by definition, where $\sqcup$ denotes disjoint union, the last assertion implies that $\mathcal E_n$ is empty.

Recall that if $(\alpha_1 \cdots \alpha_n, e_r)\in\mathcal C_n(0)\subseteq \mathcal C_n$ then there is no arrow $\gamma\neq\alpha_n$ and no arrow $\beta\neq \alpha_1$ such that $\alpha_n\beta\in I$ and $\gamma\alpha_1\in I$, but in our case there is only one quadratic relation involving the arrows $\alpha_n$ and $\alpha_1$, namely $\alpha_n\alpha_1=\partial_{\alpha_2}(W)$, then $$(AP_{n}//Q_0)=\begin{cases}\mathcal C_{n}(0) &\mbox{ \ if \ } n\equiv 0\pmod{3}\\
\emptyset &\mbox{\ otherwise.}\end{cases}$$

It is clear that any pair $(\alpha_1 \cdots \alpha_n, e_r)\in\mathcal C_n(0)$ satisfies  $t^m(\alpha_1 \cdots \alpha_n, e_r)\in\mathcal C_n(0)$, then by definition of gentle complete pairs we have that $$(AP_{n}//Q_0)=\begin{cases}\mathcal G_n &\mbox{ \ if \ } n\equiv 0\pmod{3}\\
\emptyset &\mbox{\ otherwise.}\end{cases}$$

Moreover,  $\mathcal G_n$ is not empty if and only if the triangulation $\TT$ contains at least one internal triangle.

Now, we show that the basis of $k\mathcal G_{3k}/ \operatorname{Im}(1-t)$ is in bijection of the internal triangles of $\TT$. Denote by $\{\triangle_1, \cdots, \triangle_m\}$ the set of internal triangles of $\TT$. Choose an element $(\rho_i, e_{e_i})$ for each triangle $\triangle_i$.

Since the order of any element $(\rho, e_r)\in \mathcal G_{3k}$ is 3, any element $x\in k\mathcal G_{3k}$ can be written as follows
$$x=\sum_{i=0}^m\sum_{j=0}^2\lambda_{ij}t^j(\rho_i, e_{s_i})$$

where $\lambda_{ij}\in k$, $(\rho_k,e_{r_k})\neq t^j(\rho_i, e_{r_i})$ if $k\neq i$ and each $\rho_k$ is a sequence of consecutive arrows of a 3-cycle arising from an internal triangle $\triangle_k$ of $\TT$.

In order to prove that $\Ge{3k}=m$ it is enough to show that $\operatorname{dim}_k\operatorname{ker}(1-t)=m$.

Let $x$ be an element in $\operatorname{ker}(1-t)$, then

\begin{eqnarray*}
0=(1-t)(x)&=& \sum_{i=0}^m\sum_{j=0}^2\lambda_{ij}t^j(\rho_i, e_{s_i}) - \sum_{i=0}^m\sum_{j=0}^2\lambda_{ij}t^{j+1}(\rho_i, e_{s_i})\\
&=& \sum_{i=0}^m\sum_{j=0}^2(\lambda_{ij}-\lambda_{i j-1})t^j(\rho_i, e_{s_i})
\end{eqnarray*}

then $\lambda_{i0}=\lambda_{ij}$ for $j=1,2$. Let $x_i$ be the sum $\sum_{j=0}^2t^j(\rho_i, e_{s_i})$, then $x$ can be rewritten as follows:

$$x=\sum_{i=0}^m\lambda_{i0}(\sum_{j=0}^2t^j(\rho_i, e_{s_i}))=\sum_{i=0}^m\lambda_{i0} x_j$$
then the elements $x_1, \cdots, x_m$ generates the vector space $\operatorname{ker}(1-t)$ and it is clear that they are linearly independent,  then $\operatorname{dim}_k\operatorname{ker}(1-t)=m$.
\end{proof}

We are now able to prove our main result.

\begin{proof}[Proof of Theorem]
Since in Lemma \ref{boundaries0} and Lemma \ref{boundary1} we establish a bijection between $^-(0,0)^-_i$ and the set of boundaries $\mathcal B_i$  for $i=0$ and $i=1$, then the expressions for the dimension of $\Hh{0}$ and $\Hh{1}$ follow from Theorem \ref{TeoLM} and the fact that $(Q_1//Q_0)=\emptyset$ because there are no loops in the quiver $Q_{\TT}$.

For $n\geq 2$, because $^-(0,0)_n^-$ and $\mathcal E_n$ are empty sets, the expression of the dimension of $\Hh{n}$ given in Theorem \ref{TeoLM} is reduced as follows:

\begin{equation}\label{Hh}
\Hh{n} =a \Ge{n} +b\Ge{n-1}
\end{equation}

where

$$(a,b)=\begin{cases} (1,0) & \mbox{if } n\geq 2, \ n \mbox{ even}, \operatorname{char} k \neq 2\\
(0,1) & \mbox{if } n\geq 2, \ n \mbox{ odd}, \operatorname{char} k \neq 2\\
(1,1)& \mbox{if } n\geq 2,  \operatorname{char} k=2
\end{cases}
$$

By Lemma \ref{complete}, we have that 

\begin{equation} \label{interiores}
\Ge{n}=
\begin{cases}
\abs{\operatorname{Int}(\TT)} & \mbox{if }  n\equiv 0\pmod{3}\\
0 & \mbox{otherwise}
\end{cases}
\end{equation}

Therefore the final expression of the $\Hh{n}$ depends on the value of $n$ module $3$, the characteristic of $k$ and the parity of $n$. Even so, the analysis of each case is very similar.

We analyze first the case when $\operatorname{char} k \neq 2$. If $n\equiv 0 \pmod{3}$, we have $\Ge{n}=\abs{\operatorname{Int}(\TT)}$ and $\Ge{n-1}=0$, then 

\begin{equation*} \label{interiores}
\Hh{n}=
\begin{cases}
\abs{\operatorname{Int}(\TT)} & n \mbox{ even }\\
0 & n \mbox{ odd}
\end{cases}
\end{equation*}

Observe that if $n\equiv 0 \pmod{3}$ the parity of $n$ can be described in terms of the value of $n$ module $6$, namely:
\begin{itemize}
\item if $n$ is even, then $n\equiv 0\pmod{6}$,
\item if $n$ is odd, then $n\equiv 3\pmod{6}$.
\end{itemize}

In the same way, if $n\equiv 1\pmod{3}$, the parity of $n$ is described as follows:
\begin{itemize}
\item if $n$ is even, then $n\equiv 4\pmod{6}$,
\item if $n$ is odd, then $n\equiv 1\pmod{6}$.
\end{itemize}

and we have that $\Ge{n}=0$ and $\Ge{n-1}=\abs{\operatorname{Int}(\TT)}$, then 

$$\Hh{n}=\begin{cases}
\abs{\operatorname{Int}(\TT)} & \mbox{ if $n\equiv 1\pmod 6$ }\\
0 & \mbox{ if $n\equiv 4 \pmod 6$.}
\end{cases}$$

If $n\equiv 2 \pmod{3}$, the parity of $n$ is not important because we have $\Ge{n}=0$ and $\Ge{n-1}=0$, then $\Hh{n}=0$.

Finally if $\operatorname{char} k=2$, the parity of $n$ is not important either, because $a=b=1$,  then $\Hh{n}$ depends only on  the value of $n$ module $3$. In a way similar to the case $\operatorname{k}\neq 2$  we can show  

$$\Hh{n}=\begin{cases}
\abs{\operatorname{Int}(\TT)} & \mbox{ if $n\equiv 0,1\pmod 3 $}\\
0 & \mbox{ otherwise.}
\end{cases}$$

To show that the non-trivial structure of the Gerstenhaber algebra $\operatorname{HH}^*(A_\TT)$ depends only on the existence of internal triangles, it is enough to recall that the set $\mathcal G_{n}$ is not empty if and only if $n\equiv 0\pmod{3}$ and $\TT$ contains at least one internal triangle (see Lemma \ref{complete}), then the statement follows by Theorem \ref{TeoLM}.
\end{proof}

We conclude this section with an example showing the computation of the dimension of the Hochschild cohomology groups.

\begin{exam}
Consider the triangulated surface $(S,M,\TT)$ of the Figure \ref{3boundaries}. Observe that the boundaries $B_1$ and $B_2$ are of type 1, while the boundary $B_3$ is of type 0, then $\abs{\mathcal B_0}=1$ and $\abs{\mathcal B_1}=2$. There are 3 internal triangles: $\triangle_1(\tau_1, \tau_5, \tau_7)$, $\triangle_2(\tau_1,\tau_6,\tau_2)$ and $\triangle_3(\tau_6,\tau_5,\tau_4)$. Then, according to our main result, the dimension of the Hochschild cohomology groups of $A_\TT$ is computed as follows:

\begin{itemize}
\item[a)] $\Hh{0}=1+\abs{\mathcal B_0}=2$
\item[b)] $\Hh{1}=1+ \abs{\mathcal B_1}+\abs{Q_1}-\abs{Q_0}=1+2+11-7=7$
\item[c)] If $\operatorname{char} k\neq 2$ and $n\geq 2$, then, since $\operatorname{int}(\TT)=3$, we have 
$$\Hh{n}=
\begin{cases}
3 & \mbox{ if \ } n\equiv 0,1 \pmod{6}\\
0 & \mbox{otherwise}
\end{cases}$$
\item[d)] If $\operatorname{char} k= 2$ and $n\geq 2$, then, since $\operatorname{int}(\TT)=3$, we have 
$$\Hh{n}=
\begin{cases}
3 & \mbox{ if \ } n\equiv 0,1 \pmod{6}\\
0 & \mbox{otherwise}
\end{cases}$$
\end{itemize}
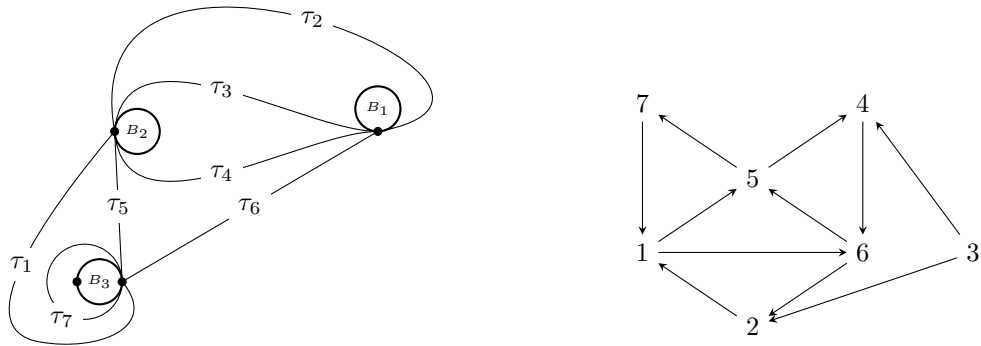
\begin{figure}[ht!]
\centering
\subfloat{
\begin{tikzpicture}
\draw[thick](-0.2,1.5) circle(0.3cm)
(-.7,-.5) circle(0.3cm)
(3,1.8)circle(0.3cm);
\node at(-0.2,1.5){\tiny{$B_2$}};
\node at(-.7,-.5) {\tiny{$B_3$}};
\node at(3,1.8){\tiny{$B_1$}};
\draw (-.9,-.5) circle(0.5cm);
\tikzset{every node/.style={fill=white}};
\draw[black] (-.5,1.5) .. controls +(85:1.5cm) and +(180:1cm).. node[midway]{$\tau_3$}(3,1.5);
\draw[black] (-.5,1.5) .. controls +(-85:1.5cm) and +(-180:1cm).. node[midway]{$\tau_4$}(3,1.5);
\draw[black] (-.5,1.5) .. controls +(230:1.5cm) and +(170:1cm).. node[midway]{$\tau_1$}(-1.5,-1.3);
\draw(-1.5,-1.3) to [out=-10, in=-50, looseness=1.5] (-.4,-.5);
\draw[black] (-.5,1.5) .. controls +(100:3.5cm) and +(10:3cm).. node[midway]{$\tau_2$}(3,1.5);
\draw(-.5,1.5) to  node[midway]{$\tau_5$}(-.4,-.5);
\draw(-.4, -.5) to node[midway]{$\tau_6$}(3,1.5);
\node at (-1.2, -1){$\tau_7$};
\filldraw [black] (-.5,1.5) circle (1.5pt)
					(3,1.5) circle (1.5pt)
					(-.4,-.5)circle(1.5pt)
					(-1,-.5)circle(1.5pt);
\end{tikzpicture}
}
 \subfloat{
  \begin{tikzpicture}[scale=1.2]
 \matrix (m)[matrix of math nodes, row sep=1.5em,column sep=3em,ampersand replacement=\&]
{7  \&  \&  4\&  \\
    \& 5\&   \&  \\
  1 \&  \&  6\& 3\\
    \& 2\&   \&  \\
};
\path[-stealth]
	(m-1-1) edge (m-3-1)
	(m-3-1) edge (m-2-2) edge (m-3-3)
	(m-2-2) edge(m-1-1) edge (m-1-3)
	(m-3-3) edge(m-2-2) edge(m-4-2)
	(m-4-2)edge(m-3-1)
	(m-1-3)edge(m-3-3)
	(m-3-4) edge (m-1-3) edge(m-4-2);
\end{tikzpicture}
}
\caption{A triangulated surface of genus $0$ and three boundaries}
\label{3boundaries}
\end{figure}

Since $\TT$ contains $3$ internal triangles, the cup product defined in $\oplus_{k\geq0}\operatorname{HH}^k(\Lambda)$ is non trivial; and if $\operatorname{char} k=0$, then the Lie bracket is also non trivial.

\end{exam}

\section{Consequences}

As mentioned before, Ladkani computed the dimension of the Hochschild cohomology groups of gentle algebras using the invariant defined by Avella-Alaminos and Geiss (for short AG-invariant).

The AG-invariant is a function $\phi_{A_\TT}:\mathbb N^2\longrightarrow \mathbb N$ depending on the ordered pairs generated by a certain algorithm. The number $\phi_{A_\TT}(n,m)$ counts how often each pair $(n,m)$ is formed in the algorithm.  See \cite{AG08} for a complete definition of AG-invariant.

The statement of Ladkani is the following:

\begin{thm1}\cite{Lad12b}[Corollary 1]\label{TeoLad}
Let $A$ be a gentle algebra. Define $\psi_A(n)=\sum_{d\mid n} \phi_A(0,d)$ for $n\geq 1$. Then
\begin{itemize}
\item[(a)] $\Hh{0}=1+ \phi_A(1,0)$.
\item[(b)] $\Hh{1}=1+\abs{Q_1}-\abs{Q_0}+\phi_A(1,1)+ 
\begin{cases}
\phi_A(0,1) & \mbox{if } \operatorname{char}(k)=2\\
0& \mbox{otherwise}
\end{cases}$.
\item[(c)] $\Hh{n}=\phi_A(1,n)+a_n\psi_A(n)+b_n\psi_A(n-1)$ for $n\geq 2$, where
$$(a_n,b_n)=
\begin{cases}
(1,0)& \mbox{if } \operatorname{char}k\neq 2 \mbox{and $n$ is even}\\
(0,1)&\mbox{if } \operatorname{char}k\neq 2 \mbox{and $n$ is odd}\\
(1,1)&\mbox{if } \operatorname{char}k=2
\end{cases}
$$
\end{itemize}
\end{thm1}

In the case of Jacobian algebras from unpunctured surfaces (more general for surfaces algebras) this invariant was computed by David-Roesler and Schiffler in \cite{DRS12} in terms of internal triangles and boundary components. We include the statement only for Jacobian algebras. 

\begin{thm1}\cite{DRS12}[Theorem 4.6]\label{TeoDRS}
Let $A_\TT$ be the Jacobian algebras from a triangulated unpunctured surfaces $(S,M,\TT)$. Then the AG-invariant $\phi_{A_\TT}$ can be computed as follows:
\begin{itemize}
\item[(a)] $\phi_{A_\TT}(0,3)=\abs{\operatorname{int}(\TT)}$.
\item[(c)] If $m\neq 3$, then $\phi_{A_\TT}(0,m)=0$.
\item[(b)] If $n\neq 0$, then the ordered pairs $(n,m)$ in the algorithm of the AG-invariant are in bijection with the boundary components of $S$ in the following way:
if $C$ is a boundary component, then the corresponding ordered pair $(n,m)$ is given by $n=\abs{n(C,\TT)}$ the cardinality of the set $n(C,\TT)$ of  marked points in $C$ that are incident to at least one arc in $\TT$ and $m=\abs{m(C,\TT)}$ the cardinality of the set $M(C,\TT)$ of boundary segments on $C$ that have both points in $n(C,\TT)$.
\end{itemize}
\end{thm1}

\begin{corollary}
The computation of Ladkani restricted to the case of Jacobian algebras from triangulated surfaces coincides with the main result of this work. 
\end{corollary}

\begin{proof}
Using Theorem \ref{TeoDRS} and Theorem \ref{TeoLad}, we compute the dimension of Hochschild cohomology for a Jacobian algebra $A_\TT$ from triangulated surface $(S,M, \TT)$:

\begin{itemize}
\item[(i)] A boundary of type 0 depicted in the Figure \ref{Frontera0} corresponds to the pair $(1,0)$, then $\phi_{A_\TT}(1,0)=\abs{\mathcal B_0}$.
\item[(ii)] A boundary of type 1 depicted in the Figure \ref{Frontera} corresponds to the pair $(1,1)$, then $\phi_{A_\TT}(1,1)=\abs{\mathcal B_1}$.
\item[(iii)] There are no pairs $(0,1)$, then $\phi_{A_\TT}(0,1)=0$.

\item[(iv)] There are no pairs $(1,n)$ with $n\leq 2$, otherwise there exists a boundary component $C$ with one marked point which is incident to at least one arc in $\TT$ but $n$  boundary segments on $C$, then $\phi_{A_\TT}(1,n)=0$.
\item[(v)] Since $\phi_A(0,d)=
\begin{cases}
\abs{\operatorname{int}(\TT)} & \mbox{if } d=3\\
0 & \mbox{otherwise}
\end{cases}$ for a Jacobian algebra $A$ from unpunctured surface, the auxiliary function $\psi_A(n)$, defined in Theorem \ref{TeoLad}, depends on $n\equiv 0\pmod{3}$, then $$\psi_A(n)=\begin{cases}
\abs{\operatorname{int}(\TT)} & \mbox{if } n\equiv 0\pmod{3}\\
0 & \mbox{otherwise.}
\end{cases}$$
\end{itemize}

Then, it is clear that both expressions, the ones in our main result and the ones in Theorem \ref{TeoLad} are the same.
\end{proof}

We conclude this section showing that the Hochschild cohomology invariant is not always complete for Jacobian algebras from triangulates surfaces. In case $(S,M)$ is a unpunctured disc, Bustamante and Gubitosi showed in \cite{BG14}[Theorem 1.2] that any two Jacobian algebras $A_\TT$ and $A_{\TT'}$  from two triangulated disc $(S,M, \TT)$ and $(S,M,\TT')$ are derived equivalent if and only if  $\operatorname{HH}^*(A_\TT)\cong \operatorname{HH}^*(A_{\TT'})$ and $\abs{Q_0(A_\TT)}=\abs{Q_0(A_{\TT'})}$. However, the following example show that the Hochschild cohomology loses information, and then it is not always  a good derived invariant.

\begin{exam}
Consider the triangulation $\TT_1$ and $\TT_2$ depicted in Figure \ref{conterexample} of a Torus with two boundaries components and 6 marked points. We show that $A_{\TT_1}$ and $A_{\TT_2}$ are not derived equivalent, but $\operatorname{HH}^*(A_{\TT_1})\cong \operatorname{HH}^*(A_{\TT_2})$ and $\abs{Q_0(A_{\TT_1})}=\abs{Q_0(A_{\TT_2})}$.

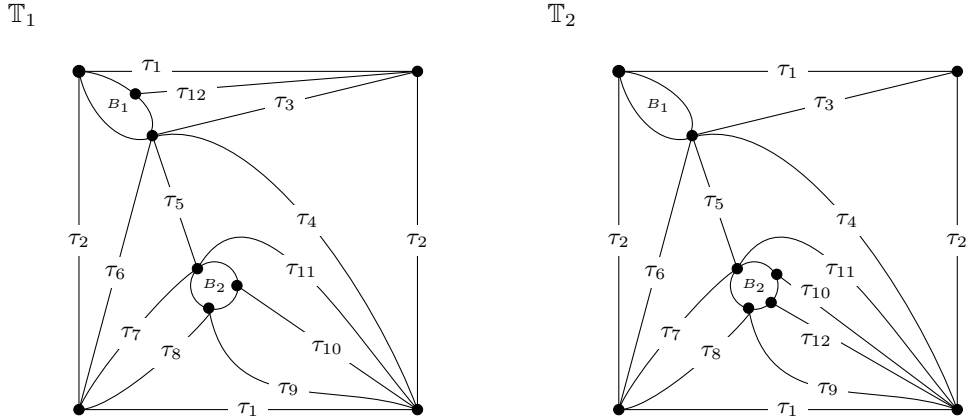
\begin{figure}[ht!]
\centering
\subfloat{
\begin{tikzpicture}[scale=1.5]
\node at (-.5,3.5){$\TT_1$};
\draw(1.2,1.1) circle(6pt);
\node at (.35,2.7){\tiny{$B_1$}};
\node at	 (1.2,1.1){\tiny{$B_2$}};
\draw (.6, 2.4) .. controls +(190:0.3cm) and +(280:0.3cm)..(0,3);
\draw (.6, 2.4) .. controls +(50:0.34cm) and +(0:0.3cm)..(0,3);
\tikzset{every node/.style={fill=white}};
\draw [name path = 1] (0,0) to node[midway]{\small{$\tau_1$}} (3,0)
	  [name path=2]	(3,0) to node[midway]{\small{$\tau_2$}}(3,3)
		(3,3) to (0,3)
		(0,3) to node[midway]{\small{$\tau_2$}} (0,0)
		(.5,2.8) to (3,3);
\node at (0.65, 3.05){\small{$\tau_1$}};
\node at (1, 2.8){\small{$\tau_{12}$}};
\draw[name path=5](.65,2.43) to node[midway]{\small{$\tau_5$}}(1.05,1.25)
	(3,0) to node[midway]{\small{$\tau_{10}$}} (1.4,1.1);
\draw[name path=6](.65,2.43) to node[midway]{\small{$\tau_6$}}(0,0)
	 [name path=3](.65,2.43) to node[midway]{\small{$\tau_3$}}(3,3);
\draw(0,0) .. controls +(0:.4cm) and +(30:.3cm).. node[midway]{\small{$\tau_8$}}(1.15,.9);
\draw(3,0) .. controls +(130:1cm) and +(60:1cm)..node[midway]{\small{$\tau_{11}$}}(1.05,1.25);
\draw(3,0) .. controls +(160:.8cm) and +(280:1cm)..node[midway]{\small{$\tau_9$}}(1.15,.9);
\draw(3,0).. controls +(110:1cm) and +(20:1cm).. node[midway]{\small{$\tau_4$}}(.6,2.4);
\draw(0,0)..controls +(70:.3cm) and +(210:.3cm)..node[midway]{\small{$\tau_7$}}(1.05,1.25);
\filldraw [black] (0,3) circle(1.5pt)
				(.65,2.43) circle(1.3pt)
				(.5,2.8) circle(1.3pt)
				(1.4,1.1) circle(1.3pt)
				(1.05,1.25) circle(1.3pt)
				(1.15,.9) circle(1.3pt)
				(3,0) circle(1.3pt)
				(0,0) circle(1.3pt)
				(3,3) circle(1.3pt);
\end{tikzpicture}
}
\hspace{1cm}
\subfloat{
\begin{tikzpicture}[scale=1.5]
\node at (-.5,3.5){$\TT_2$};
\draw(1.2,1.1) circle(6pt);
\node at (.35,2.7){\tiny{$B_1$}};
\node at	 (1.2,1.1){\tiny{$B_2$}};
\draw (.6, 2.4) .. controls +(190:0.3cm) and +(280:0.3cm)..(0,3);
\draw (.6, 2.4) .. controls +(50:0.34cm) and +(0:0.3cm)..(0,3);
\tikzset{every node/.style={fill=white}};
\draw [name path = 1] (0,0) to node[midway]{\small{$\tau_1$}} (3,0)
	  [name path=2]	(3,0) to node[midway]{\small{$\tau_2$}}(3,3)
		(3,3) to node[midway]{\small{$\tau_1$}}(0,3)
		(0,3) to node[midway]{\small{$\tau_2$}} (0,0)
		(3,0) to (1.35,.95);
\draw[name path=5](.65,2.43) to node[midway]{\small{$\tau_5$}}(1.05,1.25)
	(3,0) to (1.4,1.2);
\draw[name path=6](.65,2.43) to node[midway]{\small{$\tau_6$}}(0,0)
	 [name path=3](.65,2.43) to node[midway]{\small{$\tau_3$}}(3,3);
\node at (1.75, 1.05){\small{$\tau_{10}$}}; 
\node at (1.75,.62){\small{$\tau_{12}$}};
\draw(0,0) .. controls +(0:.4cm) and +(30:.3cm).. node[midway]{\small{$\tau_8$}}(1.15,.9);
\draw(3,0) .. controls +(130:1cm) and +(60:1cm)..node[midway]{\small{$\tau_{11}$}}(1.05,1.25);
\draw(3,0) .. controls +(160:.8cm) and +(280:1cm)..node[midway]{\small{$\tau_9$}}(1.15,.9);
\draw(3,0).. controls +(110:1cm) and +(20:1cm).. node[midway]{\small{$\tau_4$}}(.6,2.4);
\draw(0,0)..controls +(70:.3cm) and +(210:.3cm)..node[midway]{\small{$\tau_7$}}(1.05,1.25);
\filldraw [black] (0,3) circle(1.5pt)
				(.65,2.43) circle(1.3pt)
				(1.4,1.2) circle(1.3pt)
				(1.35,.95) circle(1.3pt)
				(1.05,1.25) circle(1.3pt)
				(1.15,.9) circle(1.3pt)
				(3,0) circle(1.3pt)
				(0,0) circle(1.3pt)
				(3,3) circle(1.3pt);
\end{tikzpicture}
}
\caption{Two triangulation of a torus with two boundaries}
\label{conterexample}
\end{figure}

According our main result, the computation of the Hochschild cohomology dimensions of $A_{\TT_1}$ or $A_{\TT_2}$ depend only on \emph{the number of internal triangles and the existence of boundaries of type $0$ or $1$ in $\TT_1$ (or $\TT_2$ respectively) and the number of arrows and vertices of the quiver $Q_{\TT_1}$}(or $Q_{\TT_2}$ respectively). We show that for both triangulations $\TT_1$ and $\TT_2$ those numbers coincide.

First of all, observe that both triangulations have four internal triangles: $\triangle_1(\tau_2,\tau_3,\tau_5)$, $\triangle_2(\tau_4,\tau_5, \tau_{11})$, $\triangle_3(\tau_5,\tau_6,\tau_7)$ and $\triangle_4(\tau_8, \tau_1,\tau_9)$ and neither of them have boundaries of type $0$ or $1$.

Recall that the number of vertices of a Jacobian algebra $A_{\TT}$ from a triangulated surface $(S,M,\TT)$ is equal to $6g+3b+c-6$, where $g$ is the genus of $S$, $b$ is the number of boundaries components and $c=\abs{M}$. Since, $\TT_1$ and $\TT_2$ are triangulations of a torus with two boundaries components and six marked points, we have that
$$\abs{(Q_{\TT_1})_0}=\abs{(Q_{\TT_2})_0}=12.$$

As mentioned before in Remark \ref{HH}, the number of arrows of $(Q_{\TT_i})_1$ for $i=1,2$ can be computed considering the number of internal triangles and the number of triangles with exactly one arc being part of a boundary component. In both cases, the triangulations have four internal triangles and eight triangles of the second type, then $$\abs{(Q_{\TT_1})_1}=\abs{(Q_{\TT_2})_1}=20.$$


Then it is clear that $\operatorname{HH}^n(A_{\TT_1})=\operatorname{HH}^n(A_{\TT_2})$ for every $n\in \mathbb N$ and we already showed that $\abs{(Q_{\TT_1})_0}=\abs{(Q_{\TT_2})_0}$ as we claim.

It was proven by Avella-Alaminos and Geiss that any two derived equivalent algebras $A$ and $A'$ have the same AG-invariant. Then in order to prove that $A_{\TT_1}$ and $A_{\TT_2}$ are not derived equivalent we show that $\phi_{A_{\TT_1}}\neq \phi_{A_{\TT_2}}$.

We compute the AG-invariant for $A_{\TT_1}$ and $A_{\TT_2}$ using Theorem \ref{TeoDRS}. Since the boundary components $B_1$ and $B_2$ of $(S,M,\TT_1)$ have three marked points being incident to at least one arc and there are three boundary segments between them, we have $\phi_{A_{\TT_1}}(3,3)=2$. However, there are no boundary components in $(S,M,\TT_2)$ with exactly three marked being incident to at least one arc, then $\phi_{A_{\TT_2}}(3,3)=0$. Therefore, the algebras $A_{\TT_1}$ and $A_{\TT_2}$ are not derived equivalent.
\end{exam}

\begin{corollary}
The derived class of Jacobian algebras from a unpunctured surface $(S,M)$ is not always completely determined by the Hochschild cohomology
\end{corollary}

\begin{acknowledgements}
The author was partially supported by Post-doctoral scholarship of CONICET and by NSERC of Canada. She is grateful to Ibrahim Assem and Sonia Trepode for several, interesting and helpful discussions. The author is deeply grateful to D\'epartement de Math\'ematiques of the Universit\'e de Sherbrooke and Ibrahim Assem for providing warming and ideal working conditions during her stay at Sherbrooke.
\end{acknowledgements}
\section*{References}

\bibliographystyle{elsarticle-num}
\bibliography{refe}

\end{document}